\documentclass[11pt]{amsart}
\usepackage{graphicx,amssymb}

\usepackage{amsmath,amsthm,amsfonts,amssymb,latexsym,mathrsfs,color,extarrows}

\usepackage{longtable}

\usepackage{tikz}

\usepackage[square,sort&compress,comma,numbers]{natbib}
\usepackage{nicefrac,xcolor,upref,version}

\usepackage[colorlinks=true,citecolor=cyan,backref=page]{hyperref}

\usepackage{amscd,amsthm,amsmath,amssymb,amsfonts}

\usepackage{mathrsfs}

\newtheorem{theorem}{Theorem}[section]
\newtheorem{lemma}[theorem]{Lemma}

\numberwithin{equation}{section}

\def\Q{{\mathbb {Q}}}
\def\Z{{\mathbb Z}}  
\def\R{{\mathbb R}} 
\def\C{{\mathbb C}} 
\def\eps{{\varepsilon}}

\def\beq{\begin{equation}}
\def\eeq{\end{equation}}

\def\Norm{{\rm Norm}}

\def\bfx{{\bf x}}

\begin{document}

\title[Effective approximation to complex algebraic numbers]{Effective approximation to complex algebraic numbers
by quadratic numbers}

\author{Prajeet Bajpai}
\address{Department of Mathematics, University of British Columbia, Vancouver, B.C., V6T 1Z2 Canada}
\email{pbajpai@cs.ubc.ca}

\author{Yann Bugeaud}
\address{I.R.M.A., UMR 7501, Universit\'e de Strasbourg
et CNRS, 7 rue Ren\'e Descartes, 67084 Strasbourg Cedex, France}
\address{Institut universitaire de France}
\email{bugeaud@math.unistra.fr}

\begin{abstract}
We establish an effective improvement on the Liouville inequality 
for approximation to complex non-real algebraic numbers by quadratic complex 
algebraic numbers. 
\end{abstract}


\maketitle

\section{Introduction} \label{intro}

Throughout this paper, the (na\"\i ve) height $H(\alpha )$ of an algebraic number $\alpha$
is the (na\"\i ve) height $H(P)$ of its minimal defining polynomial $P(X)$ over $\Z$, that is, the 
maximum of the absolute values of the coefficients of $P(X)$. 
The present note is a follow-up of \cite{BajBu24}, where we investigated 
the effective approximation to complex algebraic numbers
by algebraic numbers of degree at most $4$. 
Our previous results were restricted to totally complex numbers (that is, complex algebraic numbers having no real 
Galois conjugates). Here, we focus our attention on quadratic approximation to 
complex non-real algebraic numbers having at least one real 
Galois conjugate. 

Let $\xi$ be a complex non-real algebraic number of degree $d \ge 2$. By a Liouville-type argument 
(see e.g. \cite{Gu67} or below the statement of Theorem \ref{thquadr}), 
there exists an effectively computable, positive number $c_1(\xi)$ such that
\beq  \label{liouvC}
|\xi - \alpha| > c_1 (\xi) H(\alpha)^{- \frac{d}{2}},  \quad  \hbox{for every quadratic complex number $\alpha \not= \xi$}.
\eeq
When $d = 2, 3$, this is best possible; see e.g. \cite[Proposition 10.2]{BuEv09}. 
For $d \ge 4$, it follows from \cite[Corollary 2.4]{BuEv09} (see also \cite[Lemma 4.7]{BajBu24})
that  \eqref{liouvC} is best possible (up to the value of $c_1(\xi)$)
if and only if $\Q(\xi)$ is a quartic CM-field (a number field $K$ is a CM-field if it is a quadratic extension $K/F$ 
such that the image of every complex embedding of $F$ 
is contained in $\R$, but there is no complex embedding of $K$ whose image is contained in $\R$). 
In the case where $\Q(\xi)$ is not a quartic CM-field, \eqref{liouvC} can be considerably improved by means of the 
Schmidt Subspace Theorem. Namely, for every $\eps > 0$,  \cite[Corollary 2.4]{BuEv09}
asserts that there exists a positive number $c_2 (\xi, \eps)$ such that 
\beq  \label{SST}
|\xi - \alpha| > c_2 (\xi, \eps) H(\alpha)^{- \frac{3}{2} - \eps},  \quad  \hbox{for every quadratic complex number $\alpha \not= \xi$}.
\eeq
The current techniques do not yield an explicit value for $c_2 (\xi, \eps)$. 
Hence, there is a large gap between the effective statement \eqref{liouvC} and the ineffective statement \eqref{SST}. 
In the present note, we apply Baker's theory of linear forms in logarithms to get an 
effective improvement of \eqref{liouvC}. 
We establish the

\begin{theorem}  \label{thquadr} 
Let $\xi$ be a complex non-real algebraic number of degree $d \ge 4$. 
If $d=4$, assume furthermore 
that $\Q(\xi)$ is not a CM-field. 
Then, there exist effectively computable positive real numbers $\kappa (\xi)$ and $c(\xi)$ such that 
$$
|\xi - \alpha| > c(\xi) H(\alpha)^{ - \frac{d}{2} + \kappa(\xi) }, \quad 
\hbox{for every quadratic complex number $\alpha$.} 
$$
\end{theorem}

Under the additional assumption that $\xi$ is totally complex, Theorem \ref{thquadr} has been proved in \cite{BajBu24}. 
Both proofs use the same method, the main difference being a simple observation that we overlooked in \cite{BajBu24}. 
Let us explain it below.

The question of effective approximation to algebraic 
numbers by algebraic numbers of bounded degree 
is deeply connected to the effective resolution of norm-form equations. 
Indeed, let $\xi$ be an algebraic number of degree $d \ge 3$. 
Set $K = \Q (\xi)$. 
Put $\delta = 1$ if $\xi$ is real and $\delta = 2$ otherwise. 
Let $n$ be a positive integer with $n \le d-2$. 
Let $\alpha$ be an algebraic number of degree $n$. 
Then, its minimal defining polynomial $P_{\alpha} (X) := x_0 + x_1 X + \ldots + x_n X^n$ satisfies
$$
|P_{\alpha} (\xi)| = |x_0 + x_1 \xi + \ldots + x_n \xi^n| \ll_{\xi, n} |\xi - \alpha| \, H(P_{\alpha}) 
$$
and 
\beq   \label{normll}
1 \le |\Norm_{K / \Q} (P_{\alpha} (\xi))| \ll_{\xi, n}  |P_{\alpha} (\xi)|^{\delta} \cdot H(\alpha)^{d- \delta} 
\ll_{\xi, n}  |\xi - \alpha|^\delta \, H(\alpha)^d, 
\eeq
since $|P_{\alpha} (\xi)| = |P_{\alpha} (\overline{\xi})|$ if $\xi$ is non-real with complex conjugate $\overline{\xi}$. 
Here and below, the constants implicit in $\ll_{x, y, \ldots}$ and in $\gg_{x, y, \ldots}$ 
are positive, effectively computable and depend at most on $x, y, \ldots$   
This gives immediately that 
$$
|\xi - \alpha| \gg_{\xi, n} H(\alpha)^{-d}, \quad \hbox{ if $\xi$ is real,}
$$ 
and 
$$
|\xi - \alpha| \gg_{\xi, n} H(\alpha)^{-d/2}, \quad \hbox{if $\xi$ is complex non-real}.
$$
By \eqref{normll}, these lower bounds for $|\xi - \alpha|$ obtained by a Liouville-type 
argument can be improved (at the level of the exponent of $H(\alpha)$) as soon as we can establish a lower bound 
for $|\Norm_{K / \Q} (P_{\alpha} (\xi))|$ of the form $H(P_{\alpha})^c$,
valid for every integer polynomial $P_\alpha (X)$ of height at least $H_0$, 
for some positive, effectively computable real numbers $c$ and $H_0$, depending only on $\xi$ and $n$. 
For $n=1$ (and $\xi$ real, otherwise there is nothing to do), this has been done by 
Feldman \cite{Fe71} (see also \cite{Ba73}), who applied the theory of linear forms in logarithms, 
first developed by Alan Baker \cite{Ba66}. 

Recently,
under suitable assumptions on $\xi$, we showed in \cite{BajBu24} that there exist 
positive, effectively computable, real numbers $\delta (\xi)$ and $H_0 (\xi)$ such that 
\beq \label{normgg}
|\Norm_{K / \Q} (P(\xi))| \ge H(P)^{\delta (\xi)} 
\eeq
holds for every integer polynomial $P(X)$ of height at least $H_0 (\xi)$ and degree at most $2$ (resp., $3$, $4$). 
This implies that no roots of $P(X)$ is very close to $\xi$
and allowed us to obtain various effective results for approximation to $\xi$ 
by algebraic numbers of degree at most $2$ (resp., $3$, $4$) of the form
$$
|\xi - \alpha| \gg_{\xi} H(\alpha)^{-d/2 + \kappa(\xi)}, 
$$
for some positive, effectively computable, number $\kappa (\xi)$. 

It is important to observe that if this method can be applied to $\xi$, then it applies as well 
to {\it every} Galois conjugate of $\xi$. 
This remark explains why we required in \cite{BajBu24} the algebraic number $\xi$ to be totally complex: no  
effective result better than Liouville's inequality is known for approximation to real algebraic numbers by quadratic numbers, 
and this question is likely to be very difficult. 
However, to get an effective improvement over Liouville's inequality, 
we do not need to prove \eqref{normgg} for {\it all} integer polynomials $P_{\alpha} (X)$. Indeed,
if there is a Galois conjugate $\xi_0$ of $\xi$ such that $|P_{\alpha}(\xi_0)| < |P_{\alpha}(\xi)|$, then the trivial 
upper bound $|P_{\alpha}(\xi_0)| \ll_{\xi, n} H(P_{\alpha})$ used in \eqref{normll} can be 
replaced by $|P_{\alpha}(\xi_0)| < |P_{\alpha}(\xi)|$
and we eventually obtain 
$$
1 \ll_{\xi, n}  |\xi - \alpha|^3 \, H(\alpha)^d.
$$
This provides an effective improvement on Liouville's inequality. 
A version of this observation 
has already been used at the end of the proof of \cite[Theorem 2.3]{BajBu24}.

The next section gathers two estimates from the theory of linear forms in logarithms. 
Theorem \ref{thquadr} is established in Section \ref{S3}. 

Throughout the paper, we let $h$ denote the logarithmic Weil height and we set 
$h_* ( \cdot ) = \max\{h (\cdot ), 1\}$ and $\log_* (\cdot ) = \max\{ \log (\cdot ), 1 \}$.

\section{Auxiliary results}

We recall first a classical estimate for linear forms in complex logarithms of algebraic numbers.

\begin{theorem} \label{lflog} 
Let $n \ge 1$ be an integer. 
Let $\alpha_1, \ldots, \alpha_n, \alpha_{n+1}$ be non-zero algebraic numbers.  Let $D$ denote the degree of the algebraic number field generated by $\alpha_1, \ldots, \alpha_{n+1}$ over $\Q$. Let $b_1, \ldots, b_n$ be non-zero integers and set
$$
B = \max\{|b_1|, \ldots, |b_n|\}. 
$$
If $\alpha_1^{b_1} \ldots \alpha_n^{b_n} \alpha_{n+1}  \not= 1$, then we have
$$
\log |\alpha_1^{b_1} \ldots \alpha_n^{b_n} \alpha_{n+1} - 1|  
> - c(n, D) \, h_* (\alpha_1) \cdots h_*(\alpha_{n+1}) \, \log_* \Bigl( \frac{B}{h_* (\alpha_{n+1})} \Bigr). 
$$
\end{theorem}

\begin{proof}
See \cite{Bu18b} or \cite{WaLiv}. 
\end{proof}

We also need the following auxiliary result on multiplicative dependence relations between 
algebraic numbers. A version of it was originally proved by Loxton and van der Poorten \cite{LovdP83}.

\begin{theorem} \label{MultRel}
Let $m\ge 2$ be an integer and $\alpha_1,\ldots,\alpha_m$ be
multiplicatively dependent non-zero algebraic numbers. Let 
$\log \alpha_1, \ldots , \log \alpha_m$ be any determination of 
their logarithms. 
Let $D$ be the degree of the number field
generated by $\alpha_1,\ldots,\alpha_m$ over $\Q$. 
For $1\le j\le m$, let $A_j$ be a real number satisfying
$$
\log A_j \ge \max\Bigl\{h(\alpha_j),{|\log \alpha_j| \over D}, 1\Bigr\}.
$$
Then there exist rational integers $n_1,\ldots,n_m$, not all of which
are zero, such that 
$$
n_1 \log \alpha_1+ \cdots + n_m \log \alpha_m=0, \quad
\alpha_1^{n_1} \cdots \alpha_m^{n_m} = 1, 
$$
and
$$
|n_k|\le 
\bigl(11(m-1)D^3\bigr)^{m-1}{(\log A_1)
\cdots(\log A_m)\over \log A_k}, \quad \hbox{for $1\le k\le m$.}     
$$
\end{theorem} 

\begin{proof}
See \cite[Theorem 10.5]{Bu18b}. 
\end{proof}

\section{Proof of Theorem \ref{thquadr}} \label{S3}

Set $K = \Q(\xi)$.  
Let $\alpha$ be a quadratic complex (non-real) algebraic number.
Let $P_\alpha (X) = x_0 + x_1 X + x_2 X^2$ denote its minimal defining polynomial over $\Z$. 
Set
$$  
\Norm_{K/\Q} (x_0 + x_1\xi + x_2\xi^2 ) = m 
$$  
and 
$$
\bfx = x_0 + x_1\xi + x_2\xi^2 = P_\alpha (\xi). 
$$ 
Let $\sigma_1, \ldots , \sigma_d$ denote the $d$ embeddings of $K$ into $\C$ 
numbered in such a way that 
$$
|\sigma_1( \bfx  )|  \ge  |\sigma_2( \bfx  )| \ge \cdots \ge |\sigma_{d-1}( \bfx  )| \ge |\sigma_d( \bfx  )|. 
$$
For later use, recall that 
$$
d^{-1} \log |\sigma_1( \bfx  )|  \le h(\bfx) \le \log |\sigma_1( \bfx  )|.
$$
and observe that
$$
\sigma_1(\mathbf x) \cdot \sigma_2(\mathbf x) \cdots \sigma_d(\mathbf x) = m. 
$$
By \cite[Lemma 4.5]{BajBu24} (which follows from the proof of \cite[Proposition 4.3.12]{EvGy15}), 
there exists a unit $u$ in $K$ such that, putting $\mu = \bfx / u$, we have 
$$
C_1^{-1} |m|^{1/d} \le |\sigma (\mu) | \le C_1 |m|^{1/d}, 
$$
for every complex embedding $\sigma$ of $K$, with a constant $C_1 = C_1 (K) \ge 1$.

The next lemma is a reformulation of \cite[Theorem 4.6]{BajBu24}. 

\begin{lemma}    \label{thxmu}
Keep the above notation. 
Assume there are effectively computable positive real numbers 
$\kappa_1$ and $h_0$, depending only on $\xi$, such that 
$$
h(\bfx) \le \kappa_1 h_* (\mu)
$$
holds when $h(\bfx) > h_0$. 
Then, there exist effectively computable, positive $c, \kappa$, depending only on $\xi$, such that 
$$
|\xi - \alpha| > c \, H(\alpha)^{- \frac{d}{2} + \kappa }.
$$ 
\end{lemma}

Our aim is to prove the existence of $\kappa_1$ and $h_0$ as in Lemma \ref{thxmu}. 
Thus, without any loss of generality, we assume that 
\beq  \label{xgemu}
h(\bfx) \ge 2 h_*(\mu), 
\eeq
and also
$$
|\sigma_d (\bfx)| < 1. 
$$ 
Indeed, otherwise we get at once $|\sigma_1 (\bfx)| \le |m|$ and we obtain
$$
d h_*(\mu) \gg  \log_* |m| \ge \log |\sigma_1 (\bfx)| \ge h(\bfx). 
$$
Here and below, the constants 
implicit in $\gg$ and in $\ll$ are positive, effectively computable and depend at most on $\xi$.

Since the totally complex case has been fully addressed in \cite{BajBu24}, 
we assume that $K$ is not totally complex. 
In particular, its unit group has rank $2$. 

\medskip

{\bf Case 1. } 
Assume first that the embeddings can be numbered in such a way 
that $\sigma_d (\bfx) = \bfx = P_\alpha (\xi)$. 
It is a complex non-real number. 
Without any loss of generality, we assume that 
$\sigma_{d-1} (\bfx)$ is its complex conjugate.  
The proof then goes as in \cite{BajBu24}, except for Subsubcase $(\star)$ below. 

\medskip

{\bf Subcase $d = 4$.} 
It follows from \cite[Lemma 5.1]{BajBu24} that there are nonzero $a_0, a_1, a_2$ in 
the Galois closure of $K$ such that 
$$
\sigma_1 (\bfx) + a_0 \sigma_2 (\bfx) + a_1 \sigma_3 (\bfx) + a_2 \sigma_4 (\bfx) = 0. 
$$
Upon dividing by $\sigma_1 (\bfx)$
and recalling that $\bfx = \mu u$, we get 
$$
|1 - \gamma v| \le  \Bigl| \frac{a_1 \sigma_3 (\bfx)}{\sigma_1 (\bfx)} \Bigr| + \Bigl| \frac{a_2 \sigma_4 (\bfx)}{\sigma_1 (\bfx)} \Bigr|, 
$$
with
$$ 
\gamma = - \frac{a_0 \sigma_{2} (\mu)}{\sigma_1 (\mu)}, \quad 
v = \frac{\sigma_{2} (u)}{\sigma_1 (u)}.  
$$
We distinguish two cases. 

\medskip

{\bf Subsubcase $(\star)$. } 
We assume that $1 - \gamma v = 0$, that is, 
$$
\sigma_1 (\bfx) + a_0 \sigma_2 (\bfx) = 0.
$$
Let $\eta_1, \eta_2$ be a fundamental system of units of $\Q(\xi)$. Then, 
there are a root of unity $\zeta$ in $\Q(\xi)$ and integers $b_1, b_2$ such that 
$$
\bfx = \mu u = \mu \zeta \eta_1^{b_1} \eta_2^{b_2},
$$
and we have
\beq   \label{rel}
\frac{\sigma_1 (\mu \zeta)}{- a_0 \sigma_2 (\mu \zeta)} \, 
\Bigl( \frac{\sigma_1 (\eta_1)}{\sigma_2 (\eta_1)} \Bigr)^{b_1} \, 
\Bigl( \frac{\sigma_1 (\eta_2)}{\sigma_2 (\eta_2)} \Bigr)^{b_2}  = 1.
\eeq

By Theorem \ref{MultRel}, this implies that there exist rational integers $n_0, n_1, n_2$, not all of which
are zero, such that 
$$
\Bigl( \frac{\sigma_1 (\mu \zeta)}{- a_0 \sigma_2 (\mu \zeta)} \Bigr)^{n_0} \, 
\Bigl( \frac{\sigma_1 (\eta_1)}{\sigma_2 (\eta_1)} \Bigr)^{n_1} \, 
\Bigl( \frac{\sigma_1 (\eta_2)}{\sigma_2 (\eta_2)} \Bigr)^{n_2}  = 1,
$$
with
$$
|n_0| \ll 1, \quad |n_1|, |n_2| \ll h_*(\mu). 
$$
If $\frac{\sigma_1 (\eta_1)}{\sigma_2 (\eta_1)}$ and $\frac{\sigma_1 (\eta_2)}{\sigma_2 (\eta_2)}$ are multiplicatively 
independent, then we derive from \eqref{rel} that 
$b_1 n_0 = n_1$, $b_2 n_0 = n_2$, and 
$$
\max\{ |b_1|, |b_2| \} \ll h_*(\mu),
$$
giving $h(u) \ll h_* (\mu)$ and $h(\bfx) \ll h_* (\mu)$, as wanted. 

Otherwise, there exist non-zero integers $c_1$, $c_2$, with $|c_1|, |c_2| \ll 1$ such that
$$
\Bigl( \frac{\sigma_1 (\eta_1)}{\sigma_2 (\eta_1)} \Bigr)^{c_1} \, 
\Bigl( \frac{\sigma_1 (\eta_2)}{\sigma_2 (\eta_2)} \Bigr)^{c_2}  = 1.
$$
Combined with \eqref{rel}, this gives
$$
\Bigl( \frac{\sigma_1 (\mu \zeta)}{- a_0 \sigma_2 (\mu \zeta)} \Bigr)^{c_1} \, 
\Bigl( \frac{\sigma_1 (\eta_2)}{\sigma_2 (\eta_2)} \Bigr)^{-c_2 b_1 + c_1 b_2}  = 1.
$$
We deduce that 
$$
\max\{|b_1|, |b_2|\} \ll h \Bigl( \frac{\sigma_1 (\mu \zeta)}{- a_0 \sigma_2 (\mu \zeta)} \Bigr) \ll h_* (\mu),
$$
and we conclude as above that $h(\bfx) \ll h_* (\mu)$. 

\medskip

{\bf Subsubcase $(\star \star)$. } 
We assume that $|1 - \gamma v| > 0$. We treat this case below, simultaneously with the case $d \ge 6$. 

\medskip

{\bf Case $d \ge 6$. } 
It follows from \cite[Lemma 5.1]{BajBu24} that there are nonzero $a_0, a_1, a_2, b_0, b_1, b_2$ in 
the Galois closure of $K$  
such that 
$$
\sigma_1 (\xi^j) + \sum_{i=0}^2 a_i \sigma_{d-i} (\xi^j) = 0 
\quad
\hbox{and}
\quad
\sigma_2 (\xi^j) + \sum_{i=0}^2 b_i \sigma_{d-i} (\xi^j) = 0, 
\quad
\hbox{for $j = 0, 1, 2$}.
$$
This implies 
$$
b_0 \sigma_1 (\xi^j) - a_0 \sigma_2 (\xi^j) + \sum_{i=1}^2 (a_i b_0 - a_0 b_i) \sigma_{d-i} (\xi^j)= 0, 
\quad
\hbox{for $j = 0, 1, 2$}.
$$
Consequently we have $a_1 b_0 \not= a_0 b_1$. Thus, $a_1 \sigma_{d-1} (\bfx) + a_0 \sigma_d(\bfx)$ 
and $b_1 \sigma_{d-1} (\bfx) + b_0 \sigma_d(\bfx)$ cannot be both $0$. 
By permuting $\sigma_1$ and $\sigma_2$ if necessary, 
we assume that $a_1 \sigma_{d-1} (\bfx) + a_0 \sigma_d(\bfx)$ is nonzero. 
By linearity, we have
$$
\sigma_1 (\bfx) + \sum_{i=0}^2 a_i \sigma_{d-i} (\bfx) = 0,
$$
thus $\sigma_1 (\bfx) + a_2 \sigma_{d-2} (\bfx)$ is nonzero. Consequently, by dividing by $\sigma_1 (\bfx)$
and recalling that $\bfx = \mu u$, we get 
$$
0 < |1 - \gamma v| \le \sum_{i=0}^1 \Bigl| \frac{a_i \sigma_{d-i} (\bfx)}{\sigma_1 (\bfx)} \Bigr|, 
\quad \hbox{with}
\quad 
\gamma = - \frac{a_2 \sigma_{d-2} (\mu)}{\sigma_1 (\mu)}, v = \frac{\sigma_{d-2} (u)}{\sigma_1 (u)}. 
$$
This corresponds exactly to the inequalities obtained in the case $d=4$. So, we treat both cases simultaneously.

Since $|\sigma_{d-1} (\bfx)| = |\sigma_d (\bfx) | < 1$, this gives 
$|1 - \gamma v| \ll |\sigma_1 (\bfx) |^{-1}$, thus 
$$
\log |1 - \gamma v| \le - h(\bfx) / 2, \quad \hbox{if $h(\bfx)$ is large enough}. 
$$
As, by Dirichlet's unit theorem, the unit $v$ can be expressed as a product of a root of unity in $K$ times 
integral powers of elements of a fundamental system of units in $K$, it 
follows from Theorem \ref{lflog} that 
\beq  \label{ineq}
\frac{h(\bfx)}{2} \le - \log |1 - \gamma v| \ll  h_* (\gamma) \log_* \frac{h_*(v)}{h_* (\gamma)}.
\eeq
Furthermore, by \eqref{xgemu},  we get    
$$
h(v) \ll h(u) \ll (h(\bfx) +  h_* (\mu) ) \ll h(\bfx),  
$$
and we derive from \eqref{ineq} that 
$$
h(\bfx) \ll h_* (\gamma) \ll h_* (\mu). 
$$
It then follows from Lemma \ref{thxmu} that there exist effectively computable, positive 
real numbers $c, \kappa$, depending only on $\xi$, such that 
$$
|\xi - \alpha| > c \, H(\alpha)^{- \frac{d}{2} + \kappa },
$$
for every algebraic number $\alpha$ which is a root of $x_0 + x_1 X + x_2 X^2$, as soon as we are in Case 1. 

\medskip

{\bf Case 2. } 
Assume now that the embeddings cannot be numbered in such a way 
that $\sigma_d (\bfx) = \bfx = P_\alpha (\xi)$. Then, there is a Galois conjugate $\xi_0$ of $\xi$ such that 
$$
|P_\alpha (\xi_0)| < |P_\alpha (\xi)| = |P_\alpha ({\overline \xi})|. 
$$
Consequently, we have 
$$
1 \le m =  |\Norm_{K / \Q} (P_{\alpha} (\xi))| \ll  |P_{\alpha} (\xi)|^3 \cdot H(\alpha)^{d - 3} 
\ll  |\xi - \alpha |^3 \cdot H(\alpha)^d, 
$$
giving that $|\xi - \alpha | \gg H(\alpha)^{-d/3}  \gg H(\alpha)^{-d/2 + 2/3 }$, since $d \ge 4$.

\medskip

{\bf Conclusion. } 
We have established in both cases that 
$$
|\xi - \alpha| \gg H(\alpha)^{- \frac{d}{2} + \min\{\kappa, 2/3\} },
$$
for every complex non-real quadratic algebraic number $\alpha$. 
This completes the proof of Theorem \ref{thquadr}.

\end{document}